\documentclass[twoside,a4paper,12pt,reqno]{amsart}
\usepackage{amsfonts, amsbsy, amsmath, amssymb, latexsym}
\usepackage{mathrsfs,array}
\usepackage[top=1 in,right=1 in,bottom=1 in,left=1 in]{geometry}
\usepackage[utf8]{inputenc}
\usepackage{tikz}
\usepackage[english]{babel}
\usepackage{etoolbox}
\usepackage{faktor}
\usepackage{graphicx}
\usepackage{mathtools}
\usepackage{floatflt}
\usepackage{multicol}
\usepackage{lipsum}
\usepackage{diagbox}
\usepackage{wrapfig}
 \usepackage{relsize}
 \usepackage{booktabs}
 \usepackage{csquotes}
\usepackage{enumerate}
\usepackage{blindtext}
\usepackage{amssymb}
\usepackage{amsthm}
\usepackage{amsmath}
\usepackage{siunitx}
\usepackage{mathtools}
\usepackage[nodisplayskipstretch]{setspace}
\usepackage{xcolor}

\newcommand{\R}{\mathbb{R}}
\newcommand{\Z}{\mathbb{Z}}

\newcommand\hstar{h^\ast}
\newcommand{\Pyr}{\operatorname{Pyr}}
\newcommand{\Ehr}{\operatorname{Ehr}}
\newcommand{\ehr}{\operatorname{ehr}}

\usetikzlibrary{lindenmayersystems,arrows.meta}

\usepackage{array,tabularx}

\usepackage[english]{babel}
\newtheorem{theorem}{Theorem}
\newtheorem*{conj*}{Conjecture}

\newtheorem{proposition}[theorem]{Proposition}
\newtheorem{lemma}[theorem]{Lemma}
\newtheorem{corollary}[theorem]{Corollary}

\theoremstyle{definition}

\newtheorem{example}[theorem]{Example}
\usepackage{lipsum}

\begin{document}

\title{Inequalities for $f^*$-vectors of Lattice Polytopes}

\author{Matthias Beck}
\address{Department of Mathematics\\
         San Francisco State University\\
         San Francisco, CA 94132\\
         U.S.A.}
\email{becksfsu@gmail.com}

\author{Danai Deligeorgaki}
\address{Institutionen f\"or Matematik, KTH, SE-100 44 Stockholm, Sweden}
\email{danaide@kth.se}

\author{Max Hlavacek}
\address{Department of Mathematics, UC Berkeley, Berkeley, CA 94720, U.S.A.}
\email{mhlava@math.berkeley.edu}

\author{Jer\'onimo Valencia-Porras}
\address{University of Waterloo, Waterloo, ON, Canada}
\email{j2valenc@uwaterloo.ca}

\date{19 October 2022}



\thanks{We thank the organizers of \emph{Research Encounters in Algebraic and Combinatorial
Topics} (REACT 2021), where our collaboration got initiated. We are grateful to Michael
Joswig, Matthias Schymura and Lorenzo Venturello for helpful conversations.}

\maketitle


\begin{abstract}
    The Ehrhart polynomial $\text{ehr}_P(n)$ of a lattice polytope $P$ counts the number
of integer points in the $n$-th integral dilate of $P$. The $f^*$-vector of $P$,
introduced by Felix Breuer in 2012, is the vector of coefficients of $\text{ehr}_P(n)$
with respect to the binomial coefficient basis  $
\left\{\binom{n-1}{0},\binom{n-1}{1},...,\binom{n-1}{d}\right\}$, where $d = \dim P$.
   Similarly to $h/h^*$-vectors, the $f^*$-vector of $P$ coincides with the $f$-vector of
its unimodular triangulations (if they exist). 
   We present several inequalities that hold among the coefficients of $f^*$-vectors of polytopes.
   These inequalities resemble striking similarities with existing inequalities for the
coefficients of $f$-vectors of simplicial polytopes; e.g., the first half of the
$f^*$-coefficients increases and the last quarter decreases. 
   Even though $f^*$-vectors of polytopes are not always unimodal, there are several families of polytopes that carry the unimodality property. 
   We also show that for any polytope with a given Ehrhart $h^*$-vector, there is a polytope with the same $h^*$-vector whose $f^*$-vector is unimodal.

\end{abstract}

\section{Introduction}

For a $d$-dimensional \emph{lattice polytope} $P \subset \R^{d}$ (i.e., the convex hull of finitely many points in $\Z^d$) and a positive integer $n$, let $\ehr_P(n)$ denote the number of integer lattice points in $nP$. 
Ehrhart's famous theorem~\cite{ehrhartpolynomial} says that $\ehr_P(n)$ evaluates to a
polynomial in $n$.
Similar to the situations with other combinatorial polynomials, it is useful to express
$\ehr_P(n)$ in different bases; here we consider two such bases consisting of binomial
coefficients:
\begin{equation}\label{eq1}
  \ehr_P(n)
  \ = \ \sum_{k=0}^d h^*_k \binom{n+d-k}{d}
  \ = \ \sum_{k=0}^d f^*_k \binom{n-1}{k} \, .
\end{equation}
We call $(f^*_0, f^*_1, \dots, f^*_d)$ the \emph{$f^*$-vector} and 
$(h^*_0, h^*_1, \dots, h^*_d)$ the \emph{$h^*$-vector} of~$P$.
Stanley~\cite{stanleymagiccohenmac} proved that the $\hstar$-vector of any lattice polytope
is nonnegative (whereas the coefficients of $\ehr_P(n)$ written in the standard monomial basis can be negative).
Breuer~\cite{breuerehrhartf} proved that the $f^*$-vector of any lattice polytopal complex is
nonnegative (whereas the $\hstar$-vector of a complex
can have negative coefficients); his motivation was that various combinatorially-defined polynomials can be
realized as Ehrhart polynomials of complexes and so the nonnegativity of the $f^*$-vector
yields a strong constraint for these polynomials.

The $f^*$- and $h^*$-vector can also be defined through the \emph{Ehrhart series} of $P$:
\[
  \Ehr_P(z)
  \ := \ 1 + \sum_{n\geq1} \ehr_P(n) \, z^n
  \ = \ \frac{ \sum_{ k=0 }^d h^*_k \, z^k }{(1-z)^{d+1}} 
  \ = \ 1 + \sum_{ k=0 }^d f^*_k \left( \frac{ z }{ 1-z } \right)^{k+1} .
\]
It is thus sometimes useful to add the definition $f_{-1}^* := 1$.
The polynomial $\sum_{ k=0 }^d h^*_k \, z^k$ is the \emph{$h^*$-polynomial} of $P$, and its degree
is the \emph{degree} of~$P$.

The $f^*$- and $h^*$-vectors share the same relation as $f$- and $h$-vectors of 
polytopes/polyhe\-dral complexes, namely
\begin{equation}\label{eq2}
    \sum_{k=0}^d h_k^* \, z^k \ = \ \sum_{k=0}^{d+1}f_{k-1}^* \, z^{k}(1-z)^{d-k+1}
\end{equation}
\begin{equation}{\label{h* and f*}}
    h_k^* \ = \ \sum_{j=-1}^{k-1} (-1)^{k-j-1}\binom{d-j}{k-j-1}f_j^* 
\end{equation}
\begin{equation}{\label{f* and h*}}
    f_k^* \ = \ \sum_{j=0}^{k+1} \binom{d-j+1}{k-j+1}h_j^* \, .
\end{equation}

The (very special) case that $P$ admits a unimodular triangulation yields the strongest
connection between $f^*$/$h^*$-vectors and $f$/$h$-vectors: in this case the
$f^*$/$h^*$-vector of $P$ equals the $f$/$h$-vector of the triangulation, respectively.

\begin{example}\label{example with cube}
Let $P$ be the $2$-dimensional cube $[-1,1]^2$.
The unimodular triangulation of $P$ shown in Figure \ref{fig: small cube}, has $f$-vector
$(f_0,f_1,f_2)=(9,16,8)$, as $f_i$ counts its $i$-dimensional faces. Equivalently,
\[f^*(P)=(9,16,8) \, , \]
and one easily checks that \eqref{eq1} yields the familiar Ehrhart polynomial $\ehr_P(n) =
(2n+1)^2$.
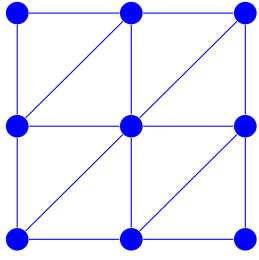
\begin{figure}
\centering
  \begin{tikzpicture}[-
  ,node distance=1.5cm,
            main node/.style={circle, blue,fill,  inner sep=3pt}]
\centering
  \node[main node] (2) { };
  \node[main node] (1) [left of=2] { };
  \node[main node] (3) [ right of=2] { };
    \node[main node] (4) [ above of=3] { };
    \node[main node] (5) [ left of=4] { };
     \node[main node] (6) [ left of=5] { };
     \node[main node] (7) [ above of=6] { };
     \node[main node] (8) [ right of=7] { };
     \node[main node] (9) [ right of=8] { };
   
  \path[every node/.style={font=\sffamily\small}]
    [color=blue] (2) edge node { } (1)
    (3) edge node { } (2)
        (2) edge node { } (5)
            (6) edge node { } (1)
                (3) edge node { } (4)
        (6) edge node { } (5)
            (5) edge node { } (4)
                (6) edge node { } (7)
                    (5) edge node { } (8)
                        (4) edge node { } (9)
        (9) edge node { } (8)
            (7) edge node { } (8)
                (5) edge node { } (1)
                    (5) edge node { } (9)
                        (2) edge node { } (4)
                            (6) edge node { } (8);
    \end{tikzpicture}
\caption{A (regular) unimodular triangulation of the cube $[-1,1]^2$.}\label{fig: small cube}
\end{figure}
\end{example}
\begin{example}\label{eg}
The $f^*$-vector of a $d$-dimensional unimodular simplex $\Delta$ equals
\[ \left[ \binom{d+1}{1},\binom{d+1}{2}, \dots, \binom{d+1} {d+1} \right],\]
coinciding with the $f$-vector of $\Delta$ considered as a simplicial complex.
If we append this vector by $f_{-1}^* = 1$, 
it gives the only instance of a \emph{symmetric} $f^*$-vector of a lattice polytope $P$, since the
equality $f_{-1}^*=f^*_d$ implies that $h_i^*=0$ for all $1\leq i \leq d$. 
\end{example}

There has been much research on (typically linear) constraints for the $h^*$-vector of a
given lattice polytope (see, e.g., \cite{stapledondelta,stapledonadditive}). On the other
hand, $f^*$-vectors seem to be much less studied, and our goal is to rectify that situation. Our motivating question is how close the $f^*$-vector of a
given lattice polytope is to being \emph{unimodal}, i.e., the $f^*$-coefficients increase
up to some point and then decrease.
Our main results are as follows.

\begin{theorem}\label{thm:main}
Let $d \geq 2$ and let $P$ be a $d$-dimensional lattice polytope. Then
\begin{enumerate}[{\rm (a)}]
\item\label{f_star_unimod_triang_first_half}
   $f^*_0<f^*_1<\cdots<f^*_{\left\lfloor \frac{d}{2}\right\rfloor-1}\leq f^*_{\left\lfloor
\frac{d}{2}\right\rfloor}$;   
\vspace{5pt}
\item\label{f_star_unimod_triang_last_quarter}
  $f^*_{ \left\lfloor \frac{3d}{4} \right\rfloor } > f^*_{ \left\lfloor \frac{3d}{4}
\right\rfloor + 1 } > \dots > f^*_d$; 
\vspace{5pt}
\item\label{f_star_0ismin}
    $f^*_k \leq f^*_{d-1-k}$ for $ 0 \leq  k \leq  \frac{(d-3)}{2}$.
\end{enumerate}
\end{theorem}
Examples \ref{example with cube} and \ref{eg} yield cases of polytopes for which the inequalities $ f^*_{ \left\lfloor \frac{3d}{4}
\right\rfloor -1 }< f^*_{ \left\lfloor \frac{3d}{4}
\right\rfloor  }$ and $f^*_{\left\lfloor
\frac{d}{2}\right\rfloor}>f^*_{\left\lfloor
\frac{d}{2}\right\rfloor+1}$
hold, respectively.

We record the following immediate consequence of Theorem \ref{thm:main}.

\begin{corollary}\label{cor:endpointslowest}
Let $P$ be a $d$-dimensional lattice polytope. Then for $0\leq k \leq d$,
\[
  f^*_{k} \geq \min \{ f^*_{0}, f^*_{d} \} \, .
\]
\end{corollary}

\begin{theorem}\label{thm:lowdimunimodal}
The $f^*$-vector of a $d$-dimensional lattice polytope, where $1\leq d \leq 13$, is
unimodal. On the other hand, there exists a $15$-dimensional lattice simplex with nonunimodal
$f^*$-vector.
\end{theorem}

Even though $f^*$-vectors are quite different from $f$-vectors of polytopes, the above
results resemble striking similarities with existing theorems on $f$-vectors.
Namely, Bj\"orner~\cite{bjornerunimodalityconj,bjornerfacenumbers,bjornerpartialunimodal} 
proved that the $f$-vector of a simplicial
$d$-polytope satisfies all inequalities in
Theorem~\ref{thm:main} 
 (with the $^*$s removed, and the last coordinate dropped). 
 In fact, Bj\"orner also showed that in the $f$-analogue of Theorem~\ref{thm:main}\eqref{f_star_unimod_triang_last_quarter} the decrease starts from $\lfloor{\frac{3(d-1)}{4}}\rfloor-1$ instead of $\lfloor \frac{3d}{4} \rfloor$, and that the inequalities in Theorem \ref{thm:main}\eqref{f_star_unimod_triang_first_half} and
\eqref{f_star_unimod_triang_last_quarter} cannot be further extended, by constructing a simplicial polytope with $f$-vector that
peaks at $f_j$, for any $\lfloor \frac d 2 \rfloor \le j \le \lfloor \frac{ 3(d-1) }{ 4 }\rfloor-1$.

Corollary~\ref{cor:endpointslowest} compares the entries of the $f^*$-vector with the minimum between the first and the last entry. Note that a similar relation for $f$-vectors of polytopes was recently proven by
Hinman~\cite{hinman}, answering a question of B\'ar\'any from the 1990s. (Hinman also
proved a stronger result, namely certain lower bounds for the ratios $\frac{ f_k }{ f_0
}$ and $\frac{ f_k }{ f_{d-1} }$.)

The $f$-analogue of Theorem~\ref{thm:lowdimunimodal} is again older:
Bj\"orner~\cite{bjornerunimodalityconj} showed that the $f$-vector of any
simplicial $d$-polytope is unimodal for $d \le 15$ (later improved to $d\leq19$ by Eckhoff~\cite{{eckhoffdimension19}}), and he and Lee~\cite{billeralee}
produced examples of 20-dimensional simplicial polytopes with nonunimodal $f$-vectors.

For a special class of polytopes we can increase the range in
Theorem~\ref{thm:main}\eqref{f_star_unimod_triang_last_quarter}.
A lattice polytope $P$ is \emph{Gorenstein of index $g$} if 
\begin{itemize}
\item $nP$ contains no interior lattice points for $1 \le n < g$,
\item $gP$ contains a unique interior lattice point, and 
\item $\ehr_{ P } (n-g)$ equals the number of interior lattice points in $nP$, for $n > g$.
\end{itemize}
This is equivalent to $P$ having degree $d+1-g$ and a symmetric $h^*$-vector (with respect to its degree).

\begin{theorem}\label{thm:gorenstein}
Let $P$ be a $d$-dimensional Gorenstein polytope of index $g$. Then
\[
   f^*_{k-1} >f^*_{k}
  \qquad \text{ for } \ \tfrac 1 2 \left( d + 1 + \left\lfloor \tfrac {d+1-g} 2 \right\rfloor
\right) \le k \le d \, .
\]
\end{theorem}

Going even further, for a certain class of polytopes we can prove unimodality of the
$f^*$-vector, a consequence of the following refinement of
Theorem~\ref{thm:main}(\ref{f_star_unimod_triang_last_quarter}) for polytopes with degree $<
\frac d 2$.

\begin{theorem}\label{h* degree}
Let $P$ be a $d$-dimensional lattice polytope with degree $\le s$. Then
  \[ f^*_{k-1} > f^*_{k} \qquad \text{ for } \ \lceil\tfrac{d+s}{2}\rceil \leq k \leq
d \, , \]
unless the degree of $P$ is $0$, i.e., $P$ is a unimodular simplex with $f^*$-vector as in Example~\ref{eg}.
\end{theorem}

This theorem implies that lattice $d$-polytopes of degree $s$ satisfying $s^2 - s - 1 \le
\frac{d}{2}$ have a
unimodal $f^*$-vector (see Proposition~\ref{proposition for pyramid} below for details).
One family with asymptotically small degree, compared to the dimension, is given by taking iterated pyramids.
Given a polytope $P \subset \R^d$, we denote by $\Pyr(P) \subset \R^{ d+1 }$ the convex hull of $P$ and the $(d+1)$st unit vector.
It is well known that $P$ and $\Pyr(P)$ have the same $h^*$-vector (ignoring an extra 0), and so we conclude:

\begin{corollary}
If $P$ is any lattice polytope then $\Pyr^n(P)$ has unimodal $f^*$-vector for sufficiently
large~$n$.
\end{corollary}


\section{Proofs}

We start with a few warm-up proofs which only use the fact that $h^*$-vectors are
nonnegative.

\begin{proof}[Proof of Theorem~\ref{thm:main}(\ref{f_star_unimod_triang_first_half})]
It follows by (\ref{f* and h*}) and the nonnegativity of $h^*(P)$ that, for $ 1 \leq k \leq \left\lfloor \frac{d}{2} \right\rfloor$,
\begin{align*}
  f_k^*- f_{k-1}^* 
  \ = \ \sum_{j=0}^{k+1} \left(\binom{d+1-j}{k+1-j}-\binom{d+1-j}{k-j}\right)h_j^*
  \ \geq   0.
\end{align*}
In fact, $ f_k^*- f_{k-1}^*$ is bounded below by $\left(\binom{d+1}{k+1}-\binom{d+1}{k}\right)h_0^*>0$ for $1\leq k < \left\lfloor \frac{d}{2}\right\rfloor$,
since $h^*_0=1$.
\end{proof}

\begin{proof}[Proof of Theorem~\ref{thm:main}(\ref{f_star_0ismin})]
For  $ 0 \leq  k \leq  \frac{(d-3)}{2}$,
equation (\ref{f* and h*}) gives
 \begin{align*}
\ & f_{d-1-k}^*- f_k^*
 \ =  \ \sum_{j=0}^{d-k} \left(\binom{d+1-j}{d-k-j}-\binom{d+1-j}{k+1-j}\right)h_j^*
 \\ &=  \ \sum_{j=0}^{d-1-2k} \left(\binom{d+1-j}{k+1}-\binom{d+1-j}{k+1-j}\right)h_j^*
  +\sum_{j=d-2k}^{d-k} \left(\binom{d+1-j}{d-k-j}-\binom{d+1-j}{k+1-j}\right)h_j^*.
 \end{align*}
We have $\binom{d+1-j}{k+1}-\binom{d+1-j}{k+1-j}\geq 0$ since
$k+1-j\leq k+1\leq \frac{d+1-j}{2}$ holds for $0\leq j \leq d-1-2k$. Similarly, $\binom{d+1-j}{d-k-j}-\binom{d+1-j}{k+1-j}\geq 0$ holds because $k+1-j\leq d-k-j\leq \frac{d+1-j}{2}$
for all $d-2k\leq j$. Therefore, it follows by the nonnegativity of $h^*$-vectors that $f_{d-1-k}^*- f_k^*\geq 0$.
\end{proof}

\begin{proof}[Proof of Theorem~\ref{h* degree}]
Since $h^*_j=0$ for $j\geq s+1$,  (\ref{f* and h*}) gives
    \begin{align*}f_{k-1}^*- f_{k}^* = \sum_{j=0}^s\left( \binom{d+1-j}{k-j}- \binom{d+1-j}{k+1-j}\right)h_j^*=\sum_{j=0}^s \frac{ 2k-d-j}{k+1-j}\binom{d+1-j}{k-j}h_j^*.\end{align*}
For $\frac{d+s}{2}\leq k \leq d$, we have $k+1-j>0$ and $2k-d-j> 0$ for all
$j=0,...,s-1$, and $k+1-j>0, \;2k-d-j\geq 0$ for $j=s$. Therefore, the claim follows by the
nonnegativity of $h^*$-vectors and the positivity of $h_0^*$.
\end{proof}

\begin{proposition}\label{proposition for pyramid} Let $P$ be a $d$-dimensional lattice polytope that has degree at most $s$ for some positive $s$.
If $d\geq2s^2-2s-2$ then the $f^*$-vector of $P$ is unimodal with a (not necessarily "sharp") peak at $f^*_p$, where $\lfloor\frac{d}{2}\rfloor\leq p\leq \lceil\frac{d+s}{2}\rceil-1$.
\end{proposition}
 
\begin{proof}
By Theorems \ref{thm:main}\eqref{f_star_unimod_triang_first_half} and \ref{h* degree}, it suffices to show that $f^*_{\lfloor \frac{d}{2}\rfloor+i}\geq f^*_{\lfloor \frac{d}{2}\rfloor+i+1}$ implies $f^*_{\lfloor \frac{d}{2}\rfloor+i+1}\geq f^*_{\lfloor \frac{d}{2}\rfloor+i+2}$, i.e., that $2f^*_{\lfloor \frac{d}{2}\rfloor+1+i}- f^*_{\lfloor \frac{d}{2}\rfloor+2+i}-f^*_{\lfloor \frac{d}{2}\rfloor+i}\geq 0$ for $0\leq i\leq\frac{s}{2}-2$.

As $h^*_j=0$ for $j\geq s+1$, by  (\ref{f* and h*}) we can express
$2f^*_{\lfloor \frac{d}{2}\rfloor+1+i}- f^*_{\lfloor \frac{d}{2}\rfloor+2+i}-f^*_{\lfloor \frac{d}{2}\rfloor+i}$ as the sum \begin{align*} &\sum_{j=0}^s \left(2\binom{d+1-j}{\lfloor \frac{d}{2}\rfloor+2-j+i}- \binom{d+1-j}{\lfloor \frac{d}{2}\rfloor+3-j+i}-\binom{d+1-j}{\lfloor \frac{d}{2}\rfloor+1-j+i}\right)h_j^*\\ 
&=\sum_{j=0}^s \left(\frac{2\left(\lceil\frac{d}{2}\rceil-i\right)}{\lfloor\frac{d}{2}\rfloor +2-j+i} -\frac{(\lceil\frac{d}{2}\rceil-i)(\lceil\frac{d}{2}\rceil-1-i)}{(\lfloor\frac{d}{2}\rfloor+2-j+i)(\lfloor\frac{d}{2}\rfloor+3-j+i)}-1 \right)\binom{d+1-j}{\lfloor \frac{d}{2}\rfloor+1-j+i}h_j^*.\end{align*}
Since $ d\geq\text{max}\{2s^2-2s-2,0\}$ 
we have that $(\lfloor\frac{d}{2}\rfloor+3-j+i)( \lfloor\frac{d}{2}\rfloor+2-j+i)$ is positive  for $j=0,...,s$ and since $h^*_j$ is nonnegative, it remains to show that 
\begin{align}& 2({\lceil\tfrac{d}{2}\rceil-i})({\lfloor\tfrac{d}{2}\rfloor +3-j+i})-{(\lceil\tfrac{d}{2}\rceil-i)(\lceil\tfrac{d}{2}\rceil-1-i)}-({\lfloor\tfrac{d}{2}\rfloor+2-j+i})({\lfloor\tfrac{d}{2}\rfloor +3-j+i}) \nonumber \\  
&\quad={d-(2j-1)(\lceil\tfrac{d}{2}\rceil-\lfloor\tfrac{d}{2}\rfloor)+4i(\lceil\tfrac{d}{2}\rceil-\lfloor\tfrac{d}{2}\rfloor) -12i+4ij-4i^2}- 6+5j -j^2 \label{second part} \\
&\quad= \begin{cases}
d - 4 i^2 + 4 i j - 12 i - j^2 + 5 j - 6 & \text{ if $d$ is even,}\\
d - 4 i^2 + 4 i j - 8 i - j^2 + 3 j - 1 & \text{ if $d$ is odd,}
\end{cases}
\nonumber
\end{align} is nonnegative for $0\leq j\leq s$.
Indeed, the conditions $j\leq s$ and $i\leq \frac{s}{2}-2$ imply that \eqref{second part} is bounded below by
\begin{align*} 
d - 4 i^2 - 12 i - j^2  - 6 
\ge d - 4 (\tfrac s 2 - 2)^2 - 12 (\tfrac s 2 - 2) - s^2  - 6
= d - 2 s^2 + 2s + 2,
\end{align*}
which is nonnegative by assumption.
\end{proof}

The next proofs use more than just the nonnegativity of $h^*$-vectors. The first
result needs the following elementary lemma on binomial coefficients.

\begin{lemma}\label{lem:binomdiff}
Let $j, k, n$ be positive integers such that $k \le n+1-j$. Then
\[
  \left| \binom n k - \binom n {k-1} \right| \ \ge \ 
  \left| \binom {n-j} k - \binom {n-j}{k-1} \right|
\]
whenever $n\neq 2k-1$.
\end{lemma}

\begin{proof}
It suffices to prove the statement for the cases $i)$ $j=1$
and the quantities $\binom n k - \binom n {k-1}$ and $\binom {n-1} k - \binom {n-1}{k-1}$
having the same sign, and $ii)$ the point when the signs change, i.e., $n = 2k$ and $j=2$.

To show case $i)$, we simplify
\begin{align*}
  \left| \binom n k - \binom n {k-1} \right|
    &= \ \frac{(n-1)!}{k!(n-k)!}\, \frac{n}{n-k+1}\left| n-2k+1 \right|
\end{align*}
and
\[
  \left| \binom {n-1} k - \binom {n-1} {k-1} \right|
  \ = \ \frac{(n-1)!}{k!(n-k)!}\left| n-2k  \right| . 
\]
If
$n \ge 2k$ then the inequalities $$\frac{n}{n-(k-1)}(n-2k+1) \ \geq \ n-2k+1 > n-2k$$ 
imply that
\begin{equation}\label{eq:j=1}
  \left| \binom n k - \binom n {k-1} \right| \ > \ \left| \binom {n-1} k - \binom {n-1} {k-1} \right|.
\end{equation}
If $n \le 2k-2$, we have $k(-2k+2+n) \le 0$ which is
equivalent to
\[
    \frac{n}{n-(k-1)}(2k-n-1) \ \geq \ 2k-n
\]
and so again~\eqref{eq:j=1} holds as a weak inequality.

To show case $ii)$, we compute
\[
  \left| \binom {2k} k - \binom{2k}{k-1} \right|
  \ = \ \frac{(2k)!}{k!(k+1)!} = \frac{(2k-2)!}{k!(k-1)!}\, \frac{2k(2k-1)}{k(k+1)}  
\]
and
\[
  \left| \binom {2k-2} {k} - \binom{2k-2}{k-1} \right|
  \ = \ \frac{(2k-2)!}{k!(k-1)!} \, .
\]
Since $2(2k-1) \ge (k+1)$ for any positive $k$, we conclude that
\[
  \left| \binom {2k} k - \binom{2k}{k-1} \right|
  \ \ge \ \left| \binom {2k-2} {k} - \binom{2k-2}{k-1} \right| . \qedhere
\]
\end{proof}

\begin{proof}[Proof of Theorem~\ref{thm:main}(\ref{f_star_unimod_triang_last_quarter})]
The inequality $f_{d-1}^*>f_d^*$ holds by Theorem~\ref{h* degree}.
Now, let $ \left\lfloor \frac{3d}{4} \right\rfloor+1\leq k < d$. By (\ref{f* and h*}),
   \begin{equation}\label{part 0}f_{k-1}^*- f_{k}^* 
   \ = \ \sum_{j=0}^{k+1} \left(\binom{d+1-j}{k-j}-\binom{d+1-j}{k+1-j}\right)h_j^* \, .\end{equation}    
The difference $\binom{d+1-j}{k-j}-\binom{d+1-j}{k+1-j}$ is nonnegative whenever $k-j\geq \lfloor \frac{d+1-j}{2} \rfloor$ and negative otherwise, i.e., the difference is nonnegative whenever $j\leq 2k-d$ and negative whenever $j> 2k-d$.
Since 
$2d-2k<2k+1-d$ for $\lfloor \frac{3d}{4} \rfloor+1\leq k $,
from \eqref{part 0} we obtain 
  \begin{align}
  f_{k-1}^*- f_{k}^*
  \ &\ge \ \sum_{j=0}^{2d-2k} \left(\binom{d+1-j}{k-j}-\binom{d+1-j}{k+1-j}\right)h_j^* \label{part 1a}
 \\ &\quad + \sum_{j=2k+1-d}^{k+1} \left(\binom{d+1-j}{k-j}-\binom{d+1-j}{k+1-j}\right)h_j^*\label{part 1b}
     \end{align}
     where the differences appearing in \eqref{part 1a} are nonnegative and the ones in
\eqref{part 1b} are negative. Our aim is to compare the sums in \eqref{part 1a} and
\eqref{part 1b} to conclude that $f_{k-1}^*- f_{k}^*$ is positive.

Using standard identities 
for binomial coefficients, 
the right hand-side of \eqref{part 1a} equals
     \begin{align*}
&\sum_{j=0}^{2d-2k} \left(\sum_{l=j}^{2d-2k-1}
\left(\binom{d-l}{k-l}-\binom{d-l}{k+1-l}\right)+
\left(\binom{2k-d+1}{3k-2d}-\binom{2k-d+1}{3k-2d+1}\right) \right)h_j^* \\
 &= \sum_{l=0}^{2d-2k-1}\left( \left(\binom{d-l}{k-l}-\binom{d-l}{k+1-l}\right)
\sum_{j=0}^{2d-2k-1-l}h_j^*\right) \\
&\qquad {} + \left(\binom{2k-d+1}{3k-2d}-\binom{2k-d+1}{3k-2d+1}\right) \sum_{j=0}^{2d-2k}h_j^*,
     \end{align*}
     hence we conclude that right hand-side of \eqref{part 1a} is bounded below by \\
     \begin{align}
&\left(\binom{d}{k}-\binom{d}{k+1}\right) h_0^* +
\left(\binom{2k-d+1}{3k-2d}-\binom{2k-d+1}{3k-2d+1}\right) \sum_{j=0}^{2d-2k}h_j^* \nonumber
        \\ &> \left(\binom{2k-d+1}{3k-2d}-\binom{2k-d+1}{3k-2d+1}\right) \sum_{j=0}^{2d-2k}h_j^*\label{part 5a}
    \end{align}
since $\binom{d}{k}-\binom{d}{k+1}>0$ for $\left\lfloor\frac{3d}{4} \right\rfloor+1\leq k < d$, 
and $h_0^*=1, h_j^*\geq 0$ for $j=1,...,2d-2k-1$.

On the other hand, for the differences appearing in \eqref{part 1b}, using that $2d-2k<j $ and $j\leq k+1$, it follows by Lemma \ref{lem:binomdiff}  that
\[ \left|
\binom{d+1-(2d -2k)}{d+1-k}-\binom{d+1-(2d -2k)}{d-k}
\right|\geq \left|\binom{d+1-j}{d+1-k}-\binom{d+1-j}{d-k}\right|,\]
i.e.,
\[ \left|
\binom{2k-d+1}{3k-2d}-\binom{2k-d+1}{3k-2d+1}
\right|\geq \left|\binom{d+1-j}{k-j}-\binom{d+1-j}{k+1-j}\right|.\] 
Hence for $j\geq 2k+1-d$,
\[
-\left(\binom{2k-d+1}{3k-2d}-\binom{2k-d+1}{3k-2d+1}\right)
\leq
\binom{d+1-j}{k-j}-\binom{d+1-j}{k+1-j} \, .\] 
Since both $-\binom{d+1-j}{k-j}+\binom{d+1-j}{k+1-j}$ and $h_j^*$ are nonnegative for
$j\geq 2k+1-d$, the sum in \eqref{part 1b} is bounded below by
     \begin{align}\label{last part} 
     -\left(\binom{2k-d+1}{3k-2d}-\binom{2k-d+1}{3k-2d+1}\right)
     \sum_{j=2k+1-d}^{d}h_j^* \, .\end{align}
Now \eqref{part 5a} and \eqref{last part} yield 
\[f_{k-1}^*- f_{k}^*> \left(\binom{2k-d+1}{3k-2d}-\binom{2k-d+1}{3k-2d+1}\right)
\left( \sum_{j=0}^{2d-2k}h_j^*-\sum_{j=2k+1-d}^{d}h_j^*\right).\]
Hibi \cite{hibi} showed that the inequality
\begin{equation}\label{eq:Hibi}   
\sum_{j=0}^{m+1}  h^{*}_j \ \geq \ \sum_{j=d-m}^d h^{*}_j
\end{equation} holds for $ m=0,...,\lfloor\frac d 2\rfloor-1$.
Since $ 2d-2k-1\leq\lfloor \frac{d}{2}\rfloor-1$ 
for $\left\lfloor \frac{3d}{4} \right\rfloor+1\leq k$, we can use \eqref{eq:Hibi}
to finally obtain
\[f_{k-1}^*- f_{k}^*>0 \, . \qedhere \]
\end{proof}

\begin{proof}[Proof of Theorem~\ref{thm:lowdimunimodal}]
If $d=1$ or $2$, there is nothing to prove.

If $3\leq d\leq 6$, then  
by Theorem \ref{thm:main},  either
\[
f^*_{0}\leq \cdots \leq f^*_{\left\lfloor \frac{d}{2}\right\rfloor} \geq f^*_{\left\lfloor \frac{3d}{4} \right\rfloor} \geq\cdots \geq f^*_d\]
or
\[
f^*_{0}\leq \cdots \leq f^*_{\left\lfloor \frac{d}{2}\right\rfloor} \leq f^*_{\left\lfloor \frac{3d}{4} \right\rfloor} \geq\cdots \geq f^*_d.\]

For $7\leq d\leq 13$,
 we will show that if 
$ f^*_{i}\geq  f^*_{i+1}, $ then $f^*_{i+1}\geq f^*_{i+2}$, for all $\left\lfloor \frac{d}{2}\right\rfloor\leq i\leq \left\lfloor \frac{3d}{4} \right\rfloor-2$. By Theorem \ref{thm:main}, this will imply the unimodality of $(f_0^*,f_{1}^*,...,f_d^*)$.

 We will examine each value of $d$ separately.
 
Suppose that $d=7$ and $f^*_{3}\geq f^*_{4}$. Then, by (\ref{f* and h*}), we compute \[2 f^*_{4}-f^*_{3}-  f^*_{5}=
14h_0^*+14h_1^*+10h_2^*+5h_3^*+h_4^*-h_5^*-h_6^*>h^*_0+h^*_1+h^*_2+h^*_3-h^*_5-h^*_6-h^*_7,\]
which is always nonnegative by \eqref{eq:Hibi}. Hence $ f^*_4-f^*_5\geq f^*_3-f^*_4$.

Likewise, for $d=8$, \eqref{eq:Hibi} implies that
 \[2f^*_{5}-f^*_{4}-f^*_6=6h_0^*+14h_1^*+14h_2^*+10h_3^*+5h_4^*+h_5^*-h_6^*-h_7^*> h^*_0+h^*_1+h^*_2+h^*_3-h^*_6-h^*_7-h^*_8\geq 0.\]

For $d=9$, we similarly get
\begin{align*}&  f^*_{5}-f^*_6-2(f^*_{4}-f^*_5)=
6h^*_0+42h^*_1+42h^*_2+28h^*_3+13h^*_4+3h^*_5-h^*_6-h^*_7>\\ &  h^*_0+h^*_1+h^*_2+h^*_3+h^*_4-h^*_6-h^*_7-h^*_8-h_9^*\geq 0\end{align*}
by \eqref{eq:Hibi}.

A similar argument works for $d=10$. By \eqref{eq:Hibi}, 
\begin{align*}&2f^*_{6}-f^*_{5}-f^*_7=
33h^*_0+48h^*_1+42h^*_2+28h^*_3+14h^*_4+4h^*_5-h^*_6-2h^*_7-h^*_8
>\\ & 2(h^*_0+h^*_1+h^*_2+h^*_3+h^*_4+h^*_5-h^*_6-h^*_7-h^*_8-h^*_9-h^*_{10})\geq0.\end{align*}
 
 For $d=11$, we need to consider two values: $i=5$ and $i=6$. The claim follows again by \eqref{eq:Hibi}, since
 \begin{align*}&  f^*_{6}-f^*_7-2(f^*_{5}-f^*_6)=
33h^*_0+132h^*_1+126h^*_2+84h^*_3+42h^*_4+14h^*_5+h^*_6-2h^*_7-h^*_8>\\ &  2(h^*_0+h^*_1+h^*_2+h^*_3+h^*_4+h^*_5-h^*_7-h^*_8-h_9^*-h^*_{10}-h_{11}^*)\geq 0,\end{align*}
and
 \begin{align*}&  f^*_{7}-f^*_8-\frac{4}{5}(f^*_{6}-f^*_7)
>
3(h^*_0+h^*_1+h^*_2+h^*_3+h^*_4+h^*_5-h^*_7-h^*_8-h^*_9-h_{10}^*-h_{11}^*)
\geq 0.\end{align*}

For $d=12$, there are also two cases: $i=6$ and $i=7$. Using \eqref{eq:Hibi}, it follows that
 \begin{align*}&  f^*_{7}-f^*_8-\frac{5}{4}(f^*_{6}-f^*_7)
> \\ &
3(h^*_0+h^*_1+h^*_2+h^*_3+h^*_4+h^*_5+h^*_6-h^*_7-h^*_8-h^*_9-h_{10}^*-h_{11}^*-h^*_{12})
\geq 0,\end{align*}
and
 \begin{align*}&  f^*_{8}-f^*_9-\frac{1}{2}(f^*_{7}-f^*_8)
> \\ &
3(h^*_0+h^*_1+h^*_2+h^*_3+h^*_4+h^*_5+h^*_6-h^*_7-h^*_8-h^*_9-h_{10}^*-h_{11}^*-h^*_{12})
\geq 0.\end{align*}

For $d=13$, we employ a stronger form of \eqref{eq:Hibi}. The expression
 \begin{align*}&  f^*_{7}-f^*_8-\frac{7}{3}(f^*_{6}-f^*_7)
\geq \\ &
3(h^*_1+h^*_2+h^*_3+h^*_4+h^*_5+h^*_6-h^*_7-h^*_8-h^*_9-h^*_{10}-h_{11}^*-h_{12}^*-h^*_{13})
\end{align*} is nonnegative by Theorem (6) in \cite{stapledondelta}.

Similarly, using Theorem (6) in \cite{stapledondelta} we have
 \begin{align*}&  2f^*_{8}-f^*_7-f^*_{9}
\geq \\ &
4(h^*_1+h^*_2+h^*_3+h^*_4+h^*_5+h^*_6-h^*_7-h^*_8-h^*_9-h^*_{10}-h_{11}^*-h_{12}^*-h^*_{13})
\geq 0.\end{align*}

To construct a polytope with nonunimodal $f^*$-vector, we employ a
family of simplices introduced by Higashitani~\cite{higashitanicounterex}.
Concretely, denote the $j$th unit vector by $e_j$ and let
\[\Delta_w \ := \ \text{conv}\big\{0,e_1,e_2,...,e_{14},w\big\}\]
where
\[ w \ := \
(\underset{\text{7}}{\underbrace{1,1,\ldots,1}},\underset{\text{7}}{\underbrace{131,131,\ldots,131}},132)
\, .\]
It has $h^*$-vector $$
(1,\underset{\text{7}}{\underbrace{0,0,\ldots,0}},131,\underset{\text{7}}{\underbrace{0,0,\ldots,0}})$$
and, via \eqref{f* and h*}, $f^*$-vector
\begin{align*}
    & (16, 120, 560, 1820, 4368, 8008, 11440, 13001,\\
    &\quad 12488, 11676, 11704, 10990, 7896, 3788, 1064, 132) \, . \qedhere
\end{align*}
\end{proof}

\begin{corollary}\label{small h*} Let $P$ be a $d$-dimensional lattice polytope such that
the $h^*$-vector of $P$ is of degree at most $5$.
Then $P$ has unimodal $f^*$-vector.
\end{corollary}
\begin{proof}
We know from Theorem \ref{thm:lowdimunimodal} that $f^*$ is unimodal when $d\leq 13$. 

Suppose that $d\geq 14$. The proof is similar to the proof of Proposition \ref{proposition for pyramid}, but we need to be a bit more precise with bounds.
By Theorems \ref{thm:main}\eqref{f_star_unimod_triang_first_half} and \ref{h* degree}, it suffices to show that  $f^*_{\lfloor \frac{d}{2}\rfloor+i}\geq f^*_{\lfloor \frac{d}{2}\rfloor+i+1}$ implies $f^*_{\lfloor \frac{d}{2}\rfloor+i+1}\geq f^*_{\lfloor \frac{d}{2}\rfloor+i+2}$, for $i=0,...,\lceil \frac{d+5}{2}\rceil-\lfloor \frac{d}{2}\rfloor-3$. Notice that $\lceil \frac{d+5}{2}\rceil=\lfloor \frac{d}{2}\rfloor+\lceil \frac{5}{2}\rceil$, hence $i=0.$ Arguing as in the proof of Proposition \ref{proposition for pyramid}, we can reduce the proof to showing that the expression in \eqref{second part} in  Proposition \ref{proposition for pyramid} is nonnegative for $0\leq j\leq 5$ and $i=0$, i.e., that
\begin{align}\label{critical
line}d-(2j-1)\left(\left\lceil\frac{d}{2}\right\rceil-\left\lfloor\frac{d}{2}\right\rfloor\right)- 6+j(5 -j)\ \geq \ 0 \, .\end{align} 
For $0\leq j\leq s$, we have
\[d-(2j-1)\left(\left\lceil\frac{d}{2}\right\rceil-\left\lfloor\frac{d}{2}\right\rfloor\right)- 6+j(5 -j)\ \geq \ d-15 \, ,\] hence \eqref{critical line} holds if $d\geq 15$. 
Finally, if $d=14$ then \eqref{critical line} holds because
\[d-(2j-1)\left(\left\lceil\frac{d}{2}\right\rceil-\left\lfloor\frac{d}{2}\right\rfloor\right)- 6+j(5 -j)\ \geq \ d-6 \, .\qedhere\]
\end{proof}



\begin{proof}[Proof of Theorem~\ref{thm:gorenstein}]
Let $s:=d+1-g$. We first consider the case that $s$ is odd; the case $s$ even will be
similar.
Since $h^*_j = 0$ for $j > s$ and $h^*_j = h^*_{ s-j }$, 
\begin{align*}
  &f_{ k-1 }^* - f_k^*
  \ = \ \sum_{j=0}^s \left( \binom{d-j+1}{k-j} - \binom{d-j+1}{k-j+1} \right) h_j^* \\
  &\quad= \ \sum_{j=0}^{ \lfloor \frac s 2 \rfloor } \left( \binom{d-j+1}{k-j} -
\binom{d-j+1}{k-j+1} \right) h_j^* + \sum_{j=\lfloor \frac s 2 \rfloor + 1}^s \left( \binom{d-j+1}{k-j} - \binom{d-j+1}{k-j+1} \right) h_j^* \\
  &\quad= \ \sum_{j=0}^{ \lfloor \frac s 2 \rfloor } \left( \binom{d-j+1}{k-j} -
\binom{d-j+1}{k-j+1} + \binom{d-s+j+1}{k-s+j} - \binom{d-s+j+1}{k-s+j+1} \right) h_j^* \, .
\end{align*}
Because we assume $k \ge \frac 1 2 (d + 1 + \lfloor \frac s 2 \rfloor)$,
\[
  \binom{d-j+1}{k-j} - \binom{d-j+1}{k-j+1} \ > \ 0 
\]
for $0 \le j \le \lfloor \frac s 2 \rfloor$.  
The inequality  \[ \binom{d-j+1}{k-j} - \binom{d-j+1}{k-j+1} + \binom{d-s+j+1}{k-s+j} -
\binom{d-s+j+1}{k-s+j+1}>0\] follows directly if
 $\binom{d-s+j+1}{k-s+j} -
\binom{d-s+j+1}{k-s+j+1}\geq 0$ or $k-s+j+1<0$.
Otherwise, Lemma~\ref{lem:binomdiff} implies that, for the same range of $j$,
\[
  \binom{d-j+1}{k-j} - \binom{d-j+1}{k-j+1} + \binom{d-s+j+1}{k-s+j} -
\binom{d-s+j+1}{k-s+j+1} \ \ge \ 0 \, .
\]
In fact, the last inequality is strict for $k \ge \frac 1 2 (d + 1 + \lfloor \frac s 2 \rfloor)$, as seen in the proof of Lemma~\ref{lem:binomdiff}.
Finally we use that $h_j^* \ge 0$ and $h_0^*=1$ to deduce that $f_{ k-1 }^* - f_k^* > 0$.

The computations in the case $s$ even is very similar. Now we write
\begin{align*}
  &f_{ k-1 }^* - f_k^*
  \ = \ \sum_{j=0}^s \left( \binom{d-j+1}{k-j} - \binom{d-j+1}{k-j+1} \right) h_j^* \\
  &\quad= \ \sum_{j=0}^{ \frac s 2 - 1 } \left( \binom{d-j+1}{k-j} - \binom{d-j+1}{k-j+1} +
\binom{d-s+j+1}{k-s+j} - \binom{d-s+j+1}{k-s+j+1} \right) h_j^* \\
  &\quad \qquad {} + \left( \binom{ d - \frac s 2 + 1 }{ k - \frac s 2 } - \binom{ d -
\frac s 2 + 1 }{ k - \frac s 2 + 1 } \right) h_{ \frac s 2 }^*
\end{align*}
and use the same argumentation as in the case $s$ odd.
\end{proof}


\section{Concluding Remarks}

There are many avenues to explore $f^*$-vectors, e.g., along analogous studies of
$h^*$-vectors, and we hope the above results form an enticing starting point.
We conclude with a few open questions which are apparent from the above.

The techniques in our proof of Theorem~\ref{thm:lowdimunimodal} do not offer much insight in the case of $14$-dimensional lattice polytopes as there are candidates $f^*$-vectors with corresponding  $h^*$-vectors that satisfy all inequalities discussed in \cite{stapledondelta}. It is unknown though if such polytopes exist.

Higashitani~\cite[Theorem 1.1]{higashitanicounterex} provided examples of $d$-dimensional
polytopes with nonunimodal $h^*$-vector for all $d\geq3$. Therefore, by
Theorem~\ref{thm:lowdimunimodal} we have examples of polytopes that have such
$h^*$-vector but their $f^*$-vector is unimodal. It would be interesting to know if the
opposite can be true, that is, if there exist polytopes with unimodal $h^*$-vector and
nonunimodal $f^*$-vector. By Corollary~\ref{small h*}, such polytopes would need to have degree at least~6.

\bibliographystyle{plain}
\bibliography{bbl}

\def\cprime{$'$} \def\cprime{$'$}
\begin{thebibliography}{10}

\bibitem{billeralee}
Louis~J. Billera and Carl~W. Lee.
\newblock A proof of the sufficiency of {McMullen}'s conditions for $f$
  -vectors of simplicial convex polytopes.
\newblock {\em J. Comb. Theory, Ser. A}, 31:237--255, 1981.

\bibitem{bjornerunimodalityconj}
Anders Bj{\"o}rner.
\newblock The unimodality conjecture for convex polytopes.
\newblock {\em Bull. Am. Math. Soc., New Ser.}, 4:187--188, 1981.

\bibitem{bjornerfacenumbers}
Anders Bj{\"o}rner.
\newblock Face numbers of complexes and polytopes.
\newblock {\em Proc. Int. Congr. Math., Berkeley/Calif. Vol. 2},
  1408-1418 , 1987.

\bibitem{bjornerpartialunimodal}
Anders Bj{\"o}rner.
\newblock Partial unimodality for {{\(f\)}}-vectors of simplicial polytopes and
  spheres.
\newblock {\em Jerusalem combinatorics '93: an international conference in
  combinatorics, May 9-17, 1993, Jerusalem, Israel}, pages 45--54. Providence,
  RI: American Mathematical Society, 1994.

\bibitem{breuerehrhartf}
Felix Breuer.
\newblock Ehrhart {$f^*$}-coefficients of polytopal complexes are non-negative
  integers.
\newblock {\em Electron. J. Combin.}, 19(4):Paper 16, 22 pp., 2012.

\bibitem{eckhoffdimension19}
J\"{u}rgen Eckhoff.
\newblock Combinatorial properties of {$f$}-vectors of convex polytopes.
\newblock {\em Normat}, 54(4):146--159, 2006.

\bibitem{ehrhartpolynomial}
Eug{\`e}ne Ehrhart.
\newblock Sur les poly\`edres rationnels homoth\'etiques \`a {$n$}\ dimensions.
\newblock {\em C. R. Acad. Sci. Paris}, 254:616--618, 1962.

\bibitem{hibi}
Takayuki Hibi.
\newblock {\em Algebraic {C}ombinatorics on {C}onvex {P}olytopes}.
\newblock Carslaw, 1992.

\bibitem{higashitanicounterex}
Akihiro Higashitani.
\newblock Counterexamples of the conjecture on roots of {E}hrhart polynomials.
\newblock {\em Discrete Comput. Geom.}, 47(3):618--623, 2012.
\newblock {\tt arXiv:1106.4633}.

\bibitem{hinman}
Joshua Hinman.
\newblock A positive answer to {Bárány's} question on face numbers of
  polytopes, 2022.
\newblock Preprint ({\tt arXiv:2204.02568}).

\bibitem{stanleymagiccohenmac}
Richard~P. Stanley.
\newblock Magic labelings of graphs, symmetric magic squares, systems of
  parameters, and {C}ohen--{M}acaulay rings.
\newblock {\em Duke Math. J.}, 43(3):511--531, 1976.

\bibitem{stapledondelta}
Alan Stapledon.
\newblock Inequalities and {E}hrhart {$\delta$}-vectors.
\newblock {\em Trans. Amer. Math. Soc.}, 361(10):5615--5626, 2009.
\newblock {\tt arXiv:math/0801.0873}.

\bibitem{stapledonadditive}
Alan Stapledon.
\newblock Additive number theory and inequalities in {E}hrhart theory.
\newblock {\em Int. Math. Res. Not.}, (5):1497--1540, 2016.
\newblock {\tt arXiv:0904.3035v2}.

\end{thebibliography}

\end{document}